\documentclass{amsart}
\usepackage{amssymb}
\usepackage[pagebackref]{hyperref}
\usepackage{color}
\newtheorem{theorem}{Theorem}[section]

\newtheorem{proposition}[theorem]{Proposition}
\theoremstyle{definition}
\newtheorem{definition}[theorem]{Definition}

\theoremstyle{remark}

\usepackage{txfonts}
\numberwithin{equation}{section}

\catcode`\ç=13
\defç{\c{c}}
\catcode`\é=13
\defé{\'e}
\catcode`\à=13
\defà{\`a}
\catcode`\è=13
\defè{\`e}
\catcode`\â=13
\defâ{\^a}
\catcode`\ù=13
\defù{\`u}
\catcode`\ê=13
\defê{\^e}
\catcode`\î=13
\defî{\^\i}
\catcode`\ô=13
\defô{\^o}

\def\1{{\rm (1)}}
\def\2{{\rm (2)}}
\def\3{{\rm (3)}}
\def\4{{\rm (4)}}
\def\5{{\rm (5)}}
\def\6{{\rm (6)}}
\def\7{{\rm (7)}}
\def\8{{\rm (8)}}

\def\c{{\rm (c)}}

\def\i{{\rm (i) }}
\def\ii{{\rm (ii) }}

\everymath{\displaystyle}
\begin{document}

\title{On Armendariz-like properties in Amalgamated algebras along ideals}

\author[Najib Mahdou]{Najib Mahdou}
\address{Department of Mathematics, Faculty of Science and Technology of Fez, Box 2202, University S. M.
Ben Abdellah Fez, Morocco}
 \email{mahdou@hotmail.com}

 \author{Abdeslam Mimouni}
\address{Department of Mathematics and Statistics, King Fahd University of Petroleum \& Minerals, P. O. Box 278, Dhahran 31261, Saudi
Arabia.} \email{ amimouni@kfupm.edu.sa }



\author[Mounir El ouarrachi]{Mounir El ouarrachi}
\address{Department of Mathematics, Faculty of Science and Technology of Fez, Box 2202, University S. M.
Ben Abdellah Fez, Morocco} \email{m.elouarrachi@gmail.com}




\thanks{This work is supported by KFUPM under Project RG 1327-1\& RG 1327-2}

\begin{abstract} Let $f: A\rightarrow B$ be a ring homomorphism and $J$ be an ideal of $B$. In this paper, we investigate the transfer of Armendariz-like properties to the amalgamation of $A$ with $B$ along $J$ with respect to $f$ (denoted by $A\bowtie^fJ)$  introduced and studied by D'Anna, Finocchiaro and Fontana in 2009. Our aim is to provide necessary and sufficient conditions for $A\bowtie^fJ,$ to be an Armendariz ring, nil-Armendariz ring and weak Armendariz ring.\smallskip
\end{abstract}

\maketitle

\section{Introduction}  All rings considered are associative with identity elements and all modules are unital. Given a ring $R$,
$nil(R)$ denotes the nil radical of $R$, that is, the set of all nilpotent elements of $R$ and the polynomial ring over $R$ is denoted by $R[x]$. For a polynomial $f(x)\in R[x]$, the content of $f(x)$, denoted by $c(f)$, is the ideal of $R$ generated by all coefficients of $f(x)$. In \cite{rc}, Rege and Chhawchharia introduced the notion of Armendariz ring as an associative ring $R$ with identity such that for every polynomials $f(x)=\displaystyle\sum_{i=0}^{i=m}a_{i}x^{i}$ and $g(x)=\displaystyle\sum_{j=0}^{i=n}b_{j}x^{j}$ in $R[x]$, $f(x)g(x)=0$ implies that $a_{i}b_{j}=0$ for every $i, j$. The name was chosen because Armendariz had shown that a reduced ring (i.e., a ring without nonzero nilpotent elements) satisfies this property (\cite{arm}). Later, in 1998, D. D. Anderson and V. Camillo continued this investigation by studying  Armendariz rings and Gauss rings (recall that a ring $R$ is said to be a Gauss ring if for every polynomials $f(x)$ and $g(x)$ in $R[x]$, $c(fg)=c(f)c(g)$). Among others, they proved that a commutative ring $R$ is Gaussian if and only if each homomorphic image of $R$ is an Armendariz ring (\cite{and}). Since then, various generalizations of Armendariz rings such as skew Armendariz ring, weak Armendariz ring, central Armendariz ring, nil-Armendariz ring etc appeared in the literature.\\

In 2006, Liu and Zhao (\cite{lz}) introduced the notion of a weak Armendariz ring as a ring $R$ such that whenever two polynomials $f(x)=\displaystyle\sum_{i=0}^{i=m}a_{i}x^{i}$ and $g(x)=\displaystyle\sum_{j=0}^{i=n}b_{j}x^{j}$ in $R[x]$ satisfy $f(x)g(x)=0$, then $a_{i}b_{j}\in nil(R)$ for every $i, j$. Among others, they proved that a ring $R$ is a weak Armendariz ring if and only if for every positive integer $n$, the $n$-by-$n$ upper triangular matrix ring $T_{n}(R)$ is a weak Armendariz ring. Moreover, if $R$ is a semicommutative ring (i.e., a ring such that whenever $ab=0$, $aRb=0$), then the polynomial ring $R[x]$ and the ring $R[x]/(x^{n})$ are weak Armendariz rings. Here, it is worth to notice that a weaker version of Armendariz ring notion also called a``weak Armendariz ring" (or $1$-Armendariz ring) is due to Lee and Wong (\cite{lw}) in the sense that whenever two linear polynomials $f(x)=a_{0}+a_{1}x$ and $g(x)=b_{0}+b_{1}x$ satisfy $fg=0$, then $a_{i}b_{j}=0$ for every $i, j=0,1$.  \\

In 2008, observing that in all examples found in the literature of Armendariz and weak Armendariz rings, the set of nilpotent elements forms an ideal, R. Antoine proved that this is not true in general and he provided an example of Armendariz ring $R$ for which $nil(R)$ is not an ideal (\cite[Example 4.8]{ant}). However, if $nil(R)$ is an ideal of $R$, then $R$ is a weak Armendariz ring, and in fact $R$ satisfies a stronger condition. This allowed him to introduce the notion of nil-Armendariz ring as a ring $R$ such that whenever two polynomials $f(x)=\displaystyle\sum_{i=0}^{i=m}a_{i}x^{i}$ and $g(x)=\displaystyle\sum_{j=0}^{i=n}b_{j}x^{j}$ in $R[x]$ satisfy $f(x)g(x)\in nil(R)[x]$, then $a_{i}b_{j}\in nil(R)$ for every $i, j$. He proved that if $R$ is a nil-Armendariz ring, then $nil(R)$ is a subring without unit of $R$. He also studied the conditions under which the polynomial ring over a nil-Armendariz ring is a nil-Armendariz ring.\\

The following diagram of implication summarizes the relation between the above notions: reduced ring $\Longrightarrow$ Armendariz ring $\Longrightarrow$
nil-Armendariz ring $\Longrightarrow$ weak Armendariz ring. The reverses of the first and second implications are not, in general, true and examples can be found in \cite[Proposition 2.1]{Hu} and \cite[Example 4.9]{ant}. However, we do not know so far any example of weak Armendariz ring which is not a
nil-Armendariz ring. This question was left open in \cite{ant}.\\

Let $A$ and $B$ be two rings with unity, let $J$ be an ideal of
$B$ and let $f: A\rightarrow B$ be a ring homomorphism. In this
setting, we can consider the following subring of $A\times B$:
\begin{center} $A\bowtie^{f}J: =\{(a,f(a)+j)\mid a\in A,j\in
J\}$\end{center} called \emph{the amalgamation of $A$ and $B$
along $J$ with respect to $f$}. This
construction is a generalization of \emph{the amalgamated
duplication of a ring along an ideal} (introduced and studied by
D'Anna and Fontana in \cite{A, AF1, aff2}). The interest of amalgamation
resides, partly, in its ability to cover several basic constructions in commutative algebra,
including pullbacks and trivial ring extensions (also called Nagata's idealizations)(cf. \cite[page 2]{Nagata}). Moreover, other
classical constructions (such as the $A+XB[X]$, $A+XB[[X]]$, and the $D+M$ constructions) can be studied as particular cases of the amalgamation (\cite[Examples 2.5 and 2.6]{aff1}) and other classical constructions, such as the CPI extensions (in the sense of Boisen and Sheldon \cite{Boisen}) are strictly related to it (\cite[Example 2.7 and Remark 2.8]{aff1}). In \cite{aff1}, the authors studied the basic properties of this construction (e.g., characterizations for $A\bowtie^{f}J$ to be a Noetherian ring, an integral domain, a reduced ring) and they characterized those distinguished pullbacks that can be expressed as an amalgamation. Moreover, in \cite{aff2}, they pursued the investigation on the structure of the rings of the form $A\bowtie^{f}J$, with particular attention to the prime spectrum, chain properties and Krull dimension.\\

This paper aims at studying the transfer of the notions of ``Armendariz ring'', ``nil-Armendariz ring" and ``weak Armendariz ring" to the amalgamation of algebras along ideals. It contains, in addition to the Introduction, three sections and each section deals respectively with one of the pre-mentioned notions.
The main results (Theorem~\ref{Armendariz.1}, Theorem~\ref{nil-Armendariz.1} and Theorem~\ref{weak Armendariz.1}) can be summarized as follows:\\

\begin{theorem} Let $(A,B)$ be a pair of rings, $f: A \rightarrow B$ be a ring homomorphism and $J$ be a  proper ideal of $B$.
\begin{enumerate}
  \item If $A\bowtie^{f}J$ is an Armendariz (resp. a nil-Armendariz, resp. a weak Armendariz) ring, then $A$ is an Armendariz (resp. a nil-Armendariz, resp. a weak Armendariz) ring.
  \item If $A$ and $f(A)+J$ are Armendariz (resp. nil-Armendariz, resp. weak Armendariz) rings, then $A\bowtie^{f}J$ is an Armendariz (resp. a nil-Armendariz, resp. a weak Armendariz) ring.
  \item Assume that $J\cap S\neq\varnothing$ where $S$ is the set of regular central elements of $B$. Then
  $A\bowtie^{f}J$ is an Armendariz (resp. a nil-Armendariz, resp. a weak Armendariz) ring if and only if $A$ and $f(A)+J$ are Armendariz (resp. nil-Armendariz, resp. weak Armendariz) rings.
  \item Assume that $J\cap nil(B)=(0)$ (resp. $J\subseteq nil(B)$). Then
  $A\bowtie^{f}J$ is an Armendariz (resp. a nil-Armendariz, resp.  a weak Armendariz) ring  if and only if $A$ is an Armendariz (resp. a nil-Armendariz, resp. a weak Armendariz) ring.
  \item Assume that $f^{-1}(J)\cap nil(A)=(0)$ (resp. $ f^{-1}(J)\subseteq nil(A)$). If $f(A)+J$ is an Armendariz (resp. a weak Armendariz) ring, then $A\bowtie^{f}J$  is an Armendariz (resp. a weak Armendariz) ring, and the equivalence holds for nil-Armendariz property.
  \item Assume that $f$ is injective.\\
  \i $f(A)\cap J=\{0\}$. Then $A\bowtie^{f}J$ is a weak Armendariz ring if and only if $f(A)+J$ is a weak Armendariz ring.\\
  \ii $J\subseteq nil(B)$. If $f(A)+J$ is a weak Armendariz, then $A\bowtie^{f}J$ is a weak Armendariz ring.
  \item Assume that $J$ is semicommutative. If $A$ is a weak Armendariz ring, then so is $A\bowtie^{f}J$.
  \item Assume that $f^{-1}(J)$ is semicommutative. If $f(A)+J$ is a weak Armendariz ring, then so is  $A\bowtie^{f}J$.
\end{enumerate}
\end{theorem}
It is worth to mention that the proofs of some assertions of the above theorem are very similar, and for the convenience of the reader, we separate the three notions in three sections and we omitted the similar proofs to avoid repetitions as much as possible.
\begin{definition}
\1 A ring $R$ is called a reduced ring if it has no non-zero nilpotent elements.\\
\2 A ring $R$ is called a semicommutative ring if for every $a, b\in R$, $ab=0$ implies that $aRb=0$.\\
\3 A ring $R$ is called an Armendariz ring if whenever polynomials $f(x)= a_{0}+a_{1}x+...+a_{n}x^{n}$,
 $g(x)= b_{0}+b_{1}x+...+b_{m}x^{m}$ in $R[x]$ satisfy $f(x)g(x)=0$, then $a_{i}b_{j}=0$ for each $i, j$.\\
\4  A ring $R$ is called a nil-Armendariz ring if whenever the product of two polynomials
$f(x)=\sum_{i=0}^{i=n}a_{i}x^{i}$  and  $g(x)=\sum_{j=0}^{j=m}b_{j}x^{j}$  in $R[x]$ satisfies  $f(x)g(x)\in nil(R)[x]$, then $a_{i}b_{j}\in nil(R)$ for each $i, j$.\\
\5  A ring $R$ is called a weak Armendariz ring if whenever the product of two polynomials
$f(x)=\sum_{i=0}^{i=n}a_{i}x^{i}$ and  $g(x)=\sum_{j=0}^{j=m}b_{j}x^{j}$  in $R[x]$  satisfies  $f(x)g(x)=0$, then $a_{i}b_{j}\in nil(R)$ for each $i, j$.
\end{definition}
\section{Armendariz property in amalgamated algebra along an ideal}\label{Armendariz}

We start this section by the following proposition which characterizes when the amalgamated algebra $A\bowtie^{f} J$ is a reduced ring.\\

\begin{proposition}(\cite[Proposition 5.4]{aff1}) Let $(A,B)$ be a pair of rings, $f: A \rightarrow B$ be a ring homomorphism and $J$ be a proper ideal of $B$. The following conditions are equivalent:
\begin{enumerate}
\item $A\bowtie^{f}J$ is a reduced ring.
\item $A$ is a reduced ring and $nil(B)\cap J=(0)$
\end{enumerate}
In particular, if $A$ and $B$ are reduced, then $A\bowtie^{f}J$ is reduced; conversely, if $J$ is a radical ideal of $B$ and $A\bowtie^{f}J$ is  reduced, then $B$ (and $A$) is reduced.
\end{proposition}
Our next Theorem states necessary and sufficient conditions under which the amalgamated algebra $A\bowtie^{f} J$ is an Armendariz ring. We notice that statements \1 and \2 are immediate consequences of the fact that Armendariz-like conditions pass trivially to subrings and finite products. For the convenience of the reader, we give simple proofs.\\
\begin{theorem}\label{Armendariz.1}
Let $(A,B)$ be a pair of rings, $f: A \rightarrow B$ be a ring homomorphism and  $J$ be a  proper ideal of $B$.
\begin{enumerate}
\item If $A\bowtie^{f}J$ is an Armendariz ring, then so is $A$.
\item If $A$ and $f(A)+J$ are Armendariz rings, then so is $A\bowtie^{f}J$.
\item  Assume that $J\cap S\neq\varnothing$ where $S$ the set of regular central elements of $B$. Then
$A\bowtie^{f}J$ is an  Armendariz ring if and only if $f(A)+J$ and $A$ are Armendariz rings.
\item Assume that $J\cap nil(B)=(0)$. Then
$A\bowtie^{f}J$ is an Armendariz ring if and only if $A$ is an Armendariz ring.
\item Assume that $ f^{-1}(J)\cap nil(A)=(0)$.
If $f(A)+J$ is an Armendariz ring, then  $A\bowtie^{f}J$ is an Armendariz ring.
 \end{enumerate}
\end{theorem}
\begin{proof}
\1 Assume that $A\bowtie^{f}J$ is Armendariz and let $f_A(x)=\sum_{i=0}^{i=n}a_{i}x^{j}$ and $g_A(x)=\sum_{j=0}^{j=m}b_{j}x^{j}$
  be two polynomials in $A[x]$ such that  $f_A(x)g_A(x)=0$. Then  for every $k\in \{0,...,n+m\};\,\,\,\sum_{i+j=k}a_{i}b_{j}=0$.
  Set $F(x)=\sum_{i=0}^{i=n}(a_{i},f(a_{i}))x^{i}$ and $G(x)=\sum_{j=0}^{j=m}(b_{j},f(b_{j}))x^{j}$. Then \\
\begin{eqnarray*}
  F(x)G(x) &=& \sum_{k=0}^{k=n+m}(\sum_{i+j=k}(a_{i}b_{j},f(a_{i}b_{j})))x^{k} \\
   &=& \sum_{k=0}^{k=n+m}(\sum_{i+j=k}a_{i}b_{j},\sum_{i+j=k}f(a_{i}b_{j}))x^{k} \\
   &=& \sum_{k=0}^{k=n+m}(\sum_{i+j=k}
   a_{i}b_{j},f(\sum_{i+j=k}a_{i}b_{j}))x^{k}.
\end{eqnarray*}
Hence $F(x)G(x)=0$ and so $(a_{i}b_{j},f(a_{i}b_{j}))=0$ since $A\bowtie^{f}J$ is Armendariz.
Thus, $a_{i}b_{j}=0$ and consequently $A$ is Armendariz.\\

\2 Assume that $A$ and $f(A)+J$ are Armendariz and let $F(x)=\sum_{i=0}^{i=n}(a_{i},f(a_{i})+j_{i})x^{i}$ and $G(x)=\sum_{j=0}^{j=m}(b_{j},f(b_{j})+k_{j})x^{j}$ be two polynomials in $(A\bowtie^{f}J)[x]$ such that $F(x)G(x)=0$.
Set $f_{B}(x)=\sum_{i=0}^{i=n}(f(a_{i})+j_{i})x^{i}$, $g_{B}(x)=\sum_{j=0}^{j=m}(f(b_{j})+k_{j})x^{j}$,
$f_A(x)=\sum_{i=0}^{i=n}a_{i}x^{j}$ and $g_A(x)=\sum_{j=0}^{j=m}b_{j}x^{j}$. Then $F(x)G(x)=0$ implies that
$f_{A}(x)g_{A}(x)=0$ and $f_{B}(x)g_{B}(x)=0$, which in turn implies that $(f(a_{i})+j_{i})(f(b_{j})+k_{j})=0$
and $a_{i}b_{j}=0$ for every $i, j$ since $f(A)+J$ and $A$ are Armendariz rings. Therefore $A\bowtie^{f}J$ is Armendariz.\\

\3 Let $S$ be the set of regular central elements of $B$. Assume that $J\cap S\neq\varnothing$ and $A\bowtie^{f}J$ is Armendariz.
Let $f_{A}(x)=\sum_{i=0}^{i=n}(f(a_{i})+j_{i})x^{i}$ and $g_{A}(x)=\sum_{j=0}^{j=m}(f(b_{j})+k_{j})x^{j}$ be two polynomials in $(f(A)+J)[x]$
such that $f_{A}(x)g_{A}(x)=0$ and let $e$ be a regular element of $J$. Set $F(x)=\sum_{i=0}^{i=n}(0,e(f(a_{i})+j_{i}))x^{i} \,\, and \,\, G(x)=\sum_{j=0}^{j=m}(0,e(f(b_{j})+k_{j}))x^{j}$. Clearly
\begin{eqnarray*}
    F(x)G(x) &=& \sum_{k=0}^{k=n+m}(\sum_{i+j=k}(0,e^{2}(f(a_{i})+j_{i})(f(b_{j})+k_{j}))x^{k} \\
     &=& \sum_{k=0}^{k=n+m}(0,e^{2}\sum_{i+j=k}(f(a_{i})+j_{i})(f(b_{j})+k_{j}))x^{k}=0.
  \end{eqnarray*}
So $(0,e(f(a_{i})+j_{i}))(0,e(f(b_{j})+k_{j}))=0$ since $A\bowtie^{f}J$ is Armendariz; which implies that
$e^{2}(f(a_{i})+j_{i})(f(b_{j})+k_{j})=0$ for every $i, j$. Hence $(f(a_{i})+j_{i})(f(b_{j})+k_{j})=0$,
 and this shows that $f(A)+J$ is Armendariz.\\

\4 Assume that $J\cap nil(B)=(0)$ and $A$ is Armendariz. Let $F(x)=\sum_{i=0}^{i=n}(a_{i},f(a_{i})+j_{i})x^{i}$ and $G(x)=\sum_{t=0}^{t=m}(b_{t},f(b_{t})+k_{t})x^{t}$ be two polynomials in $(A\bowtie^{f}J)[x]$ such that $F(x)G(x)=0$.
Set $f_{B}(x)=\sum_{i=0}^{i=n}(f(a_{i})+j_{i})x^{i}$, $g_{B}(x)=\sum_{t=0}^{t=m}(f(b_{t})+k_{t})x^{t}$,
$f_A(x)=\sum_{i=0}^{i=n}a_{i}x^{j}$ and $g_A(x)=\sum_{t=0}^{t=m}b_{t}x^{t}$. Then $F(x)G(x)=0$ implies that $f_{A}g_{A}=0$ ( and $f_{B}g_{B}=0$)
which in turn implies that $a_{i}b_{t}=0$ since $A$ is an Armendariz ring. Thus $(f(a_{i})+j_{i})(f(b_{t})+k_{t})\in J$ for every $i,t$. Next, we show that $(f(a_{i})+j_{i})(f(b_{t})+k_{t})=0$ for every $i,t$. For this, we proceed by induction on the degree $n$ of $F(x)$. If $n=0$, it is clear. Suppose that $n\geq 1$ and the induction hypothesis.\\
{\it Claim}: $(f(a_{0})+j_{0})(f(b_{t})+k_{t})=0$ for every $0\leq t \leq m$. Indeed, suppose that $\exists  t\in\{0,..., m\}$ such that $(f(a_{0})+j_{0})(f(b_{t})+k_{t})\neq 0$ and let $l$ be the smallest integer in $\{0,..., m\}$ such that $(f(a_{0})+j_{0})(f(b_{l})+k_{l})\neq 0$. Then for $ t\in \{0,..., l-1\}$, $(f(a_{0})+j_{0})(f(b_{t})+k_{t})=0$ and so $((f(b_{t})+k_{t})J(f(a_{0})+j_{0}))^{2} =0$. Thus
$(f(b_{t})+k_{t})J(f(a_{0})+j_{0})=0$ since $J\cap nil(B)=(0)$.
Hence $(f(a_{l-t})+j_{l-t})(f(b_{t})+k_{t})((f(a_{0})+j_{0})(f(b_{l})+k_{l}))^{2}=$
$(f(a_{l-t})+j_{l-t})(f(b_{t})+k_{t})(f(a_{0})+j_{0})(f(b_{l})+k_{l})(f(a_{0})+j_{0})(f(b_{l})+k_{l})$
$\in (f(a_{l-t})+j_{l-t})((f(b_{t})+k_{t})J(f(a_{0})+j_{0}))(f(b_{l})+k_{l})=0$. But since
the coefficient of the term $x^{l}$ in $f_{B}g_{B}=0$ is zero, we obtain\\
$0=(f(a_{0})+j_{0})(f(b_{l})+k_{l})+(f(a_{1})+j_{1})(f(b_{l-1})+k_{l-1})+...+(f(a_{l})+j_{l})(f(b_{0})+k_{0})
= (f(a_{0})+j_{0})(f(b_{l})+k_{l})+\sum_{t=1}^{t=l-1}(f(a_{l-t})+j_{l-t})(f(b_{t})+k_{t})$.
Multiplying $((f(a_{0})+j_{0})(f(b_{l})+k_{l}))^{2}$ to the preceding equation on the right side we obtain:
$((f(a_{0})+j_{0})(f(b_{l})+k_{l}))^{3}+\sum_{t=1}^{t=l-1}(f(a_{l-t})+j_{l-t})(f(b_{t})+k_{t})((f(a_{0})+j_{0})(f(b_{l})+k_{l}))^{2}=0$.
Hence $((f(a_{0})+j_{0})(f(b_{l})+k_{l}))^{3}=0$ and so $(f(a_{0})+j_{0})(f(b_{l})+k_{l})\in J\cap nil(B)=0$. Thus $(f(a_{0})+j_{0})(f(b_{l})+k_{l})=0$
which is a contradiction. Consequently,  $(f(a_{0})+j_{0})(f(b_{t})+k_{t})=0 $ for every $t\in \{0,..., m\}$. Now, set
$ F_{1}(x)=(f(a_{1})+j_{1})+(f(a_{2})+j_{2})x+...+(f(a_{n})+j_{n})x^{n-1}$. Then $F(x)=(a_{0}, f(a_{0})+j_{0})+xF_{1}(x)$ and by the claim, $(a_{0}, f(a_{0})+j_{0})G(x)=0$. Thus $F_{1}(x)G(x)=0$ and by the induction hypothesis, $(f(a_{i})+j_{i})(f(b_{t})+k_{t})=0 $ for every $1\leq i\leq n$ and $0\leq t\leq m$. Therefore  $(f(a_{i})+j_{i})(f(b_{t})+k_{t})=0 $ for every $0\leq i\leq n$ and $0\leq t\leq m$ and hence $(a_{i},f(a_{i})+j_{i})(b_{t},f(b_{t})+k_{t})=0$ for every $0\leq i\leq n$ and $0\leq t\leq m$. It follows that $A\bowtie^{f}J$ is Armendariz.\\

\5 Assume that $f^{-1}(J)\cap nil(A)=(0)$ and $f(A)+J$ is an Armendariz ring. Our argument is similar to the one in \4. Let $F(x)=\sum_{i=0}^{i=n}(a_{i},f(a_{i})+j_{i})x^{i}$ and $G(x)=\sum_{j=0}^{j=m}(b_{j},f(b_{j})+k_{j})x^{j}$ be two polynomials in $(A\bowtie^{f}J)[x]$ such that $F(x)G(x)=0$. Set $f_{B}(x)=\sum_{i=0}^{i=n}(f(a_{i})+j_{i})x^{i}$, $g_{B}(x)=\sum_{j=0}^{j=m}(f(b_{j})+k_{j})x^{j}$, $f_A(x)=\sum_{i=0}^{i=n}a_{i}x^{j}$ and $g_A(x)=\sum_{j=0}^{j=m}b_{j}x^{j}$. Since $F(x)G(x)=0$, $f_{A}g_{A}=0$ and $f_{B}g_{B}=0$. Thus $(f(a_{i})+j_{i})(f(b_{j})+k_{j})=0$ for every $i, j$ since  $f(A)+J$ is an Armendariz ring; and hence $a_{i}b_{j}\in f^{-1}(J)$. To show that $a_{i}b_{j}=0$ for every $i, j$,
we proceed by induction on the degree $n$ of $F(x)$. If $n=0$, this is trivial. Suppose that $n\geq 1$ and the induction hypothesis. First we show that $a_{0}b_{j}=0$ for every $0\leq j\leq m$. Indeed, suppose that $ \exists j\in \{0;...;m\}$ such that $a_{0}b_{j}\neq0$. Let $k$ be the smallest positive integer in $\{0,..., m\}$ such that $a_{0}b_{k}\neq0$. Then for $j\in \{0,..., k-1\}$, $a_{0}b_{j}=0$ and so $(b_{j}f^{-1}(J)a_{0})^{2}=0$. Then $b_{j}f^{-1}(J)a_{0}\subseteq f^{-1}(J)\cap nil(A)=(0)$ and so $b_{j}f^{-1}(J)a_{0}=0$. Hence $(a_{k-j}b_{j})(a_{0}b_{k})^{2}= a_{k-j}b_{j}a_{0}b_{k}a_{0}b_{k} \in  a_{k-j}(b_{j}f^{-1}(J)a_{0})b_{k}=0$. The coefficient of the term $x^{k}$ in $f_{A}(x)g_{A}(x)=0$ is
$0=a_{0}b_{k}+a_{1}b_{k-1}+...+a_{k}b_{0}= a_{0}b_{k}+\sum_{j=1}^{j=k-1}a_{k-j}b_{j}$. Multiplying $( a_{0}b_{k})^{2}$ to the preceding equation on the right side, we obtain $(a_{0}b_{k})^{3}+\sum_{j=1}^{j=k-1}(a_{k-j}b_{j})(a_{0}b_{k})^{2}=0$.
Hence $(a_{0}b_{k})^{3}=0$ and so $(a_{0}b_{k})\in f^{-1}(J)\cap nil(A)=(0)$, which is a contradiction. Consequently $a_{0}b_{j}=0$, for every $j\in \{0,...,m\}$. Finally, as in \4, set $ F_{1}(x)=(f(a_{1})+j_{1})+(f(a_{2})+j_{2})x+...+(f(a_{n})+j_{n})x^{n-1}$. Then $F(x)=(a_{0}, f(a_{0})+j_{0})+xF_{1}(x)$ and by the claim, $(a_{0}, f(a_{0})+j_{0})G(x)=0$. Thus $F_{1}(x)G(x)=0$ and by the induction hypothesis $a_{i}b_{j}=0$ for every $1\leq i\leq n$ and $0\leq j\leq m$. Therefore $a_{i}b_{j}=0$ for every $i, j$ and hence $(a_{i},f(a_{i})+j_{i})(b_{j},f(b_{j})+k_{j})=0$ for every $i, j$. It follows that $A\bowtie^{f}J$ is an Armendariz ring.
\end{proof}
\section{Nil-Armendariz property in amalgamated algebra along an ideal}

\begin{theorem}\label{nil-Armendariz.1}
Let $(A,B)$ be a pair of rings, $f : A \rightarrow B$ be a ring homomorphism and  $J$ be a  proper ideal of $B$, then
\begin{enumerate}
  \item If $A\bowtie^{f}J$ is  a nil-Armendariz ring, then so is $A$.
  \item If $A$ and $f(A)+J$ are nil-armendariz rings, then so is $A\bowtie^{f}J$.
  \item Assume that $J\cap S\neq\varnothing$ where $S$ is the set of regular central element of $B$. Then
    $A\bowtie^{f}J$ is a nil-Armendariz ring if and only if $f(A)+J$ and $A$ are nil-Armendariz rings.
  \item Assume that $J\subseteq nil(B)$. Then $A\bowtie^{f}J$ is a nil-Armendariz ring if and only if $A$ is a nil-Armendariz ring.
  \item Assume that $ f^{-1}(J)\subseteq nil(A)$. Then $A\bowtie^{f}J$ is a nil-Armendariz ring if and only if $f(A)+J$ is a nil-Armendariz ring.
  \item Assume that $f$ is injective.\\
  \i $f(A)\cap J=0$. Then $A\bowtie^{f}J$ is a nil-Armendariz ring if and only if $f(A)+J$ is a nil-Armendariz ring.\\
  \ii $J\subseteq nil(B)$. Then $A\bowtie^{f}J$ is a nil-Armendariz ring if and only if $f(A)+J$ is a nil-Armendariz ring.
\end{enumerate}
\end{theorem}
\begin{proof}
The proofs of the assertions \1, \2 and \3 are similar to \1, \2 and \3 in Theorem~\ref{Armendariz.1}.

\noindent \4 Suppose that $A$ is a nil-Armendariz ring. Then $\frac{A\bowtie^{f}J}{0\times J}\simeq A$ is a nil-Armendariz ring. Let $F(x)=\sum_{i=0}^{i=n}(a_{i},f(a_{i})+j_{i})x^{i}$ and $G(x)=\sum_{j=0}^{j=m}(b_{j},f(b_{j})+k_{j})x^{j}$ be two polynomials in $(A\bowtie^{f}J)[x]$ such that $F(x)G(x)=\sum_{k=0}^{k=n+m}(\sum_{i+j=k}(a_{i}b_{j},(f(a_{i})+j_{i})(f(b_{j})+k_{j})))x^{k}\in nil((A\bowtie^{f}J)[x])$.
Set $\overline{F(x)}=\sum_{i=0}^{i=n}\overline{(a_{i},f(a_{i})+j_{i})}x^{i}$ and $\overline{G(x)}=\sum_{j=0}^{j=m}\overline{(b_{j},f(b_{j})+k_{j})}x^{j}$
in $ \frac{A\bowtie^{f}J}{0\times J}[x]$. Then
$F(x)G(x)\in nil(A\bowtie^{f}J)[x]$ implies that $\overline{F(x)}\,\overline{G(x)}\in nil\frac{A\bowtie^{f}J}{0\times J}[x]$.
Consequently $\overline{(a_{i},f(a_{i})+j_{i})}\,\overline{(b_{j},f(b_{j})+k_{j})}\in nil\frac{A\bowtie^{f}J}{0\times J}$ since $\frac{A\bowtie^{f}J}{0\times J}$ is nil-Armendariz. Hence $(a_{i}b_{j},(f(a_{i})+j_{i})(f(b_{j})+k_{j}))^{p_{ij}}\in 0\times J$ for some integer $p_{ij}$. Therefore $((f(a_{i})+j_{i})(f(b_{j})+k_{j}))^{p_{ij}}\in J\subseteq nil(B)$. Hence $(a_{i},f(a_{i})+j_{i})(b_{j},f(b_{j})+k_{j})\in nil(A\bowtie^{f}J)$
and this shows that $A\bowtie^{f}J$ is nil-Armendariz.\\

\noindent \5 Assume that $f^{-1}(J)\subseteq nil(A)$ and suppose that  $A\bowtie^{f}J$ is nil-Armendariz. Let $f_{A}(x)=\sum_{i=0}^{i=n}(f(a_{i})+j_{i})x^{i}$ and $g_{A}(x)=\sum_{j=0}^{j=m}(f(b_{j})+k_{j})x^{j}$ such that $f_{A}(x)g_{A}(x)\in nil(f(A)+J)[x]$. Let $F(x)=\sum_{i=0}^{i=n}(a_{i},f(a_{i})+j_{i})x^{i}$ and $G(x)=\sum_{j=0}^{j=m}(b_{j},f(b_{j})+k_{j})x^{j}$. Since $f_{A}(x)g_{A}(x)=\sum_{k=0}^{k=n+m}(\sum_{i+j=k}(f(a_{i})+j_{i})(f(b_{j})+k_{j}))x^{k}\in nil(f(A)+J)[x]$, $\sum_{i+j=k}(f(a_{i})+j_{i})(f(b_{j})+k_{j})\in nil(f(A)+J)$ for every $k\in \{0,. . . ,n+m\}$. Thus $\sum_{i+j=k}(f(a_{i}b_{j})+t_{ij})\in nil(f(A)+J)$ with $t_{ij}\in J$. Hence, for every $k\in \{0,. . . , n+m\}$,  $f(\sum_{i+j=k}a_{i}b_{j})+\sum_{i+j=k}t_{ij}$ is nilpotent.
So $(f(\sum_{i+j=k}a_{i}b_{j}))^{n_{ij}}\in J$ for some positive integer $n_{ij}$, and therefore $(\sum_{i+j=k}a_{i}b_{j})^{n_{ij}}\in f^{-1}(J)\subseteq nil(A)$ which, in turn, implies that $\sum_{i+j=k}a_{i}b_{j}\in nil(A)$. Consequently, $F(x)G(x)\in nil( A\bowtie^{f}J)[x]$ and hence $(\sum_{i+j=k}a_{i}b_{j},\sum_{i+j=k}(f(a_{i})+j_{i})(f(b_{j})+k_{j}))\in nil( A\bowtie^{f}J)$. Since $A\bowtie^{f}J$ is a nil-Armendariz ring, $(a_{i}b_{j},(f(a_{i})+j_{i})(f(b_{j})+k_{j}))=(a_{i},f(a_{i})+j_{i})(b_{j},f(b_{j})+k_{j})$ is nilpotent and so $(f(a_{i})+j_{i})(f(b_{j})+k_{j})\in nil(f(A)+J)$. Hence $f(A)+J$ is nil-Armendariz, as desired.\\
\noindent The converse is similar to \4  by using the fact that $\frac{A\bowtie^{f}J}{f^{-1}(J)\times0}\simeq f(A)+J$. \\

\6 Assume that $f$ is injective.\\
\i $f(A)\cap J=0$. In this case $A\bowtie^{f}J\simeq f(A)+J$ and the conclusion follows.\\
\ii Assume that $J\subseteq nil(B)$ and suppose that $f(A)+J$ is nil-Armendariz. Let $F(x)=\sum_{i=0}^{i=n}(a_{i},f(a_{i})+j_{i})x^{i}$ and $G(x)=\sum_{j=0}^{j=m}(b_{j},f(b_{j})+k_{j})x^{j}$ be two polynomials in $(A\bowtie^{f}J)[x]$ such that $F(x)G(x)\in nil(A\bowtie^{f}J)[x]$.
Set $f_{B}(x)=\sum_{i=0}^{i=n}(f(a_{i})+j_{i})x^{i}$ and $g_{B}(x)=\sum_{j=0}^{j=m}(f(b_{j})+k_{j})x^{j}$. Then
$F(x)G(x)\in  nil(A\bowtie^{f}J)[x]$ implies that $f_{B}(x)g_{B}(x)\in nil(f(A)+J)[x]$.
Hence $(f(a_{i})+j_{i})(f(b_{j})+k_{j}) \in nil(f(A)+J)$ since $f(A)+J$ is nil-Armendariz.
Now, we show that $a_{i}b_{j}$ is nilpotent. Indeed, since $(f(a_{i})+j_{i})(f(b_{j})+k_{j})=(f(a_{i}b_{j})+t_{ij})\in nil(f(A)+J),  t_{ij}\in J$,
$(f(a_{i}b_{j})+t_{ij})^{n_{ij}}=0 $ for some positive integer $n_{ij}$. Therefore $(f(a_{i}b_{j})^{n_{ij}}=f((a_{i}b_{j})^{n_{ij}})\in J\subseteq nil(B)$ and
so $(f((a_{i}b_{j})^{n_{ij}}))^{m_{ij}}=0 $ for some positive integer $m_{ij}$. Hence $f(((a_{i}b_{j})^{n_{ij}})^{m_{ij}})=0$ and therefore $(a_{i}b_{j})^{n_{ij}m_{ij}}=0$ since $f$ is injective. Consequently, $(a_{i},f(a_{i})+j_{i})(b_{j},f(b_{j})+k_{j})$ is nilpotent and this shows that $A\bowtie^{f}J$ is nil-Armendariz.\\
\noindent Conversely, suppose that $A\bowtie^{f}J$ is nil-Armendariz and let $f_A(x)=\sum_{i=0}^{i=n}(f(a_{i})+j_{i})x^{i}$ and $g_A(x)=\sum_{j=0}^{j=m}(f(b_{j})+k_{j})x^{j}$ be two polynomials in $(f(A)+J)[x]$ such that $f_A(x)g_A(x)\in nil(f(A)+J)[x]$.
Set $F(x)=\sum_{i=0}^{i=n}(a_{i},f(a_{i})+j_{i})x^{i}$ and $G(x)=\sum_{j=0}^{j=m}(b_{j},f(b_{j})+k_{j})x^{j}$. Then
\begin{eqnarray*}
         F(x)G(x) &=&\sum_{k=0}^{k=n+m}(\sum_{i+j=k}(a_{i}b_{j},(f(a_{i})+j_{i})(f(b_{j})+k_{j}))x^{k}\\
         &=& \sum_{k=0}^{k=n+m}(\sum_{i+j=k}a_{i}b_{j},\sum_{i+j=k}(f(a_{i})+j_{i})(f(b_{j})+k_{j}))x^{k}
\end{eqnarray*}
 But $f_A(x)g_A(x)\in nil (f(A)+J)[x]$ implies that $(\sum_{i+j=k}(f(a_{i})+j_{i})(f(b_{j})+k_{j}))^{n_{ij}}=0$ for some positive integer $n_{ij}$.
 Thus $(\sum_{i+j=k}(f(a_{i}b_{j})+t_{ij}))^{n_{ij}}=0$ for some positive integer $n_{ij}$
and so $(f(\sum_{i+j=k}a_{i}b_{j})+\sum_{i+j=k}t_{ij})^{n_{ij}}=0$. Hence $(f(\sum_{i+j=k}a_{i}b_{j})^{n_{ij}}\in J\subseteq nil(B)$ and therefore  $f((\sum_{i+j=k}a_{i}b_{j})^{m_{ij}})=0$ for some positive integer $m_{ij}$. Since $f$ is injective, $(\sum_{i+j=k}a_{i}b_{j})^{m_{ij}}=0$ and hence $F(x)G(x)\in nil(A\bowtie^{f}J)[x]$, which in turn, implies that $(a_{i}b_{j},(f(a_{i})+j_{i})(f(b_{j})+k_{j}))\in nil(A\bowtie^{f}J)$.
Therefore $(f(a_{i})+j_{i})(f(b_{j})+k_{j})\in nil(f(A)+J)$ and this shows that $f(A)+J$ is nil-Armendariz.
\end{proof}
\section{weak Armendariz property in amalgamated algebra along an ideal}\label{weak Armendariz}

\begin{theorem}\label{weak Armendariz.1}
Let $(A,B)$ be a pair of rings, $f: A \rightarrow B$ be a ring homomorphism and  $J$ be a  proper ideal of $B$, then
\begin{enumerate}
  \item If $A\bowtie^{f}J$ is a weak Armendariz ring, then so is $A$.
  \item If $A$ and $f(A)+J$ are weak Armendariz rings, then so is $A\bowtie^{f}J$.
  \item Assume that $J\cap S\neq\varnothing$ where $S$ is the set of regular central element of $B$. Then
  $A\bowtie^{f}J$ is a weak Armendariz ring if and only if $f(A)+J$ and $A$ are weak Armendariz rings.
  \item Assume that $J\subseteq nil(B)$. Then $A$ is weak Armendariz ring if and only if $A\bowtie^{f}J$ is a weak Armendariz ring.
  \item Assume that $ f^{-1}(J)\subseteq nil(A)$. If  $f(A)+J$  is a weak Armendariz ring, then $A\bowtie^{f}J$ is a weak Armendariz ring.
  \item Assume that $f$ is injective.\\
  \i $f(A)\cap J=0$. Then $A\bowtie^{f}J$ is a weak Armendariz ring if and only if $f(A)+J$ is a weak Armendariz ring.\\
  \ii $J\subseteq nil(B)$. If $f(A)+J$ is a weak Armendariz ring, then $A\bowtie^{f}J$ is a weak Armendariz ring.
\item Assume that $J$ is semicommutative. If $A$ is a weak Armendariz ring, then so is $A\bowtie^{f}J$.
\item Assume that $f^{-1}(J)$ is semicommutative. If $f(A)+J$ is a weak Armendariz ring, then so is $A\bowtie^{f}J$.
\end{enumerate}
\end{theorem}
\begin{proof}
The assertions \1, \2 and \3 are similar to \1, \2 and \3 in Theorem~\ref{Armendariz.1}, and the assertions \4, \5 and \6 are similar to \4, \5 and \6 in Theorem~\ref{nil-Armendariz.1}.\\
\noindent \7 Assume that $J$ is semicommutative and $A$ is weak Armendariz. Let $F(x)=\sum_{i=0}^{i=n}(a_{i},f(a_{i})+j_{i})x^{i}$ and $G(x)=\sum_{j=0}^{j=m}(b_{j},f(b_{j})+k_{j})x^{j}$ be two polynomials in $A\bowtie^{f}J[x]$ such that $F(x)G(x)=0$ and set  $f_{A}(x)=\sum_{i=0}^{i=n}a_{i}x^{j}$, $g_{A}(x)=\sum_{j=0}^{j=m}b_{j}x^{j}$, $f_{B}(x)=\sum_{i=0}^{i=n}(f(a_{i})+j_{i})x^{i}$ and $g_{B}(x)=\sum_{j=0}^{j=m}(f(b_{j})+k_{j})x^{j}$. Then $F(x)G(x)=0$ implies that $f_{A}(x)g_{A}(x)=\sum_{l=0}^{l=n+m}(\sum_{i+j=l}a_{i}b_{j})x^{l}=0$
and $f_{B}(x)g_{B}(x)=\sum_{l=0}^{l=n+m}(\sum_{i+j=l}(f(a_{i})+j_{i})(f(b_{j})+k_{j}))x^{l}=0$. Hence $\sum_{i+j=l}a_{i}b_{j}=0 $ for all $l =0,1...,n+m $ and $\sum_{i+j=l}(f(a_{i})+j_{i})(f(b_{j})+k_{j})=0$ for all $l =0,1...,n+m$. Thus $a_{i}b_{j}\in nil(A)$ since $A$ is weak Armendariz, and so $(a_{i}b_{j})^{n_{ij}}=0$ for some positive integer $n_{ij}$. To show that $(f(a_{i})+j_{i})(f(b_{j})+k_{j})\in nil(f(A)+J)$, we proceed by induction on $i+j$.\\
If $i+j=0$, we have $(f(a_{0})+j_{0})(f(b_{0})+k_{0})=0\in nil(f(A)+J)$\\
Let $l$ be a positive integer such that $(f(a_{i})+j_{i})(f(b_{j})+k_{j})\in nil(f(A)+J)$ when $i+j< l$. We will show that $(f(a_{i})+j_{i})(f(b_{j})+k_{j})\in nil(f(A)+J)$ when $i+j= l$. We have $((f(a_{0})+j_{0})(f(b_{l})+k_{l}))^{n_{0l}=p}\in J$ since $(a_{0}b_{l})^{p}=0$. By the induction hypothesis, $(f(a_{0})+j_{0})(f(b_{l-1})+k_{l-1})\in nil(f(A)+J)$. Let $t$ be a positive integer such that $((f(a_{0})+j_{0})(f(b_{l-1})+k_{l-1}))^{t}=0$.
Then $((f(b_{l-1})+k_{l-1})(f(a_{0})+j_{0}))^{t+1}=0$, and hence $(((f(a_{1})+j_{1})(f(b_{l-1})+k_{l-1}))((f(a_{0})+j_{0})(f(b_{l})+k_{l}))^{p+1}(f(a_{1})+j_{1}))\times ((f(b_{l-1})+k_{l-1})(f(a_{0})+j_{0}))^{t+1}((f(b_{l-1})+k_{l-1})((f(a_{0})+j_{0})(f(b_{l})+k_{l}))^{p+1})=0$. Since $(((f(a_{1})+j_{1})(f(b_{l-1})+k_{l-1}))((f(a_{0})+j_{0})(f(b_{l})+k_{l}))^{p+1}(f(a_{1})+j_{1}))(f(b_{l-1})+k_{l-1})(f(a_{0})+j_{0})\in J$,
$((f(b_{l-1})+k_{l-1})(f(a_{0})+j_{0}))^{t}((f(b_{l-1})+k_{l-1})((f(a_{0})+j_{0})(f(b_{l})+k_{l}))^{p+1})\in J$,
$(f(b_{l})+k_{l})((f(a_{0})+j_{0})(f(b_{l})+k_{l}))^{p})(f(a_{1})+j_{1})\in J$ and $J$ is semicommutative, it follows that
$(((f(a_{1})+j_{1})(f(b_{l-1})+k_{l-1}))((f(a_{0})+j_{0})(f(b_{l})+k_{l}))^{p+1}(f(a_{1})+j_{1}))((f(b_{l-1})+k_{l-1})(f(a_{0})+j_{0}))\times
    (f(b_{l})+k_{l})((f(a_{0})+j_{0})(f(b_{l})+k_{l}))^{p})(f(a_{1})+j_{1})\times
    ((f(b_{l-1})+k_{l-1})(f(a_{0})+j_{0}))^{t}((f(b_{l-1})+k_{l-1})((f(a_{0})+j_{0})(f(b_{l})+k_{l}))^{p+1})=0$.
Hence $[((f(a_{1})+j_{1})(f(b_{l-1})+k_{l-1}))((f(a_{0})+j_{0})(f(b_{l})+k_{l}))^{p+1}]^{2}(f(a_{1})+j_{1})\times
    ((f(b_{l-1})+k_{l-1})(f(a_{0})+j_{0}))^{t}((f(b_{l-1})+k_{l-1})((f(a_{0})+j_{0})(f(b_{l})+k_{l}))^{p+1})=0$.
Continuing this procedure, we obtain $[((f(a_{1})+j_{1})(f(b_{l-1})+k_{l-1}))((f(a_{0})+j_{0})(f(b_{l})+k_{l}))^{p+1}]^{t+3}=0$.\\
Thus $((f(a_{1})+j_{1})(f(b_{l-1})+k_{l-1}))((f(a_{0})+j_{0})(f(b_{l})+k_{l}))^{p+1}\in nil(J)$.\\
Similarly, we can show that $((f(a_{i})+j_{i})(f(b_{l-i})+k_{l-i}))((f(a_{0})+j_{0})(f(b_{l})+k_{l}))^{p+1}\in nil(J)$ for $i=2,..., l$.\\
Since $J$ is semicommutative, $nil(J)$ is an ideal and consequently $\sum_{i=1}^{i=l}(f(a_{i})+j_{i})(f(b_{l-i})+k_{l-i}))((f(a_{0})+j_{0})(f(b_{l})+k_{l}))^{p+1}\in nil(J)$.
Multiplying the equation $\sum_{i+j=l}(f(a_{i})+j_{i})(f(b_{j})+k_{j})=0$ on the right side by
$((f(a_{0})+j_{0})(f(b_{l})+k_{l}))^{p+1}$, we obtain $((f(a_{0})+j_{0})(f(b_{l})+k_{l}))^{p+2}=-\sum_{i=1}^{i=l}(f(a_{i})+j_{i})(f(b_{j})+k_{j})((f(a_{0})+j_{0})(f(b_{l})+k_{l}))^{p+1}\in nil(J)$.\\
Thus $(f(a_{0})+j_{0})(f(b_{l})+k_{l})\in nil(f(A)+J)$.\\
Let $q=n_{1,l-1}$. Then $((f(a_{1})+j_{1})(f(b_{l-1})+k_{l-1}))^{q}\in J$. By analogy with the above proof, we have \\
    $$\sum_{i=2}^{i=l}(f(a_{i})+j_{i})(f(b_{l-i})+k_{l-i})((f(a_{1})+j_{1})(f(b_{l-1})+k_{l-1}))^{q+1}\in nil(J)$$
Suppose that $((f(a_{0})+j_{0})(f(b_{l})+k_{l}))^{s}=0$. Then \\
    $$((f(a_{1})+j_{1})(f(b_{l-1})+k_{l-1}))^{q+1}((f(a_{0})+j_{0})(f(b_{l})+k_{l}))^{s}((f(a_{1})+j_{1})(f(b_{l-1})+k_{l-1}))^{q+1}=0$$
Since $((f(a_{1})+j_{1})(f(b_{l-1})+k_{l-1}))^{q+1}\in J$ and $J$ is semicommutative,
    $$ ((f(a_{0})+j_{0})(f(b_{l})+k_{l})((f(a_{1})+j_{1})(f(b_{l-1})+k_{l-1}))^{q+1})^{s+1}=0$$
Therefore $$((f(a_{0})+j_{0})(f(b_{l})+k_{l})((f(a_{1})+j_{1})(f(b_{l-1})+k_{l-1}))^{q+1}\in nil(J)$$
Multiplying the equation $\sum_{i+j=l}(f(a_{i})+j_{i})(f(b_{j})+k_{j}))=0$ on the right side by $((f(a_{1})+j_{1})(f(b_{l-1})+k_{l-1}))^{q+1}$, we obtain
    $((f(a_{1})+j_{1})(f(b_{l-1})+k_{l-1}))^{q+2}=-\sum_{i=2}^{i=l}((f(a_{i})+j_{i})(f(b_{l-i})+k_{l-i}))((f(a_{1})+j_{1})(f(b_{l-1})+k_{l-1}))^{q+1}
                -((f(a_{0})+j_{0})(f(b_{l})+k_{l}))((f(a_{1})+j_{1})(f(b_{l-1})+k_{l-1}))^{q+1}\in nil(J)$.
Therefore $ (f(a_{1})+j_{1})(f(b_{l-1})+k_{l-1})\in nil(f(A)+J)$.\\
A similar argument shows that $$(f(a_{2})+j_{2})(f(b_{l-2})+k_{l-2})\in nil(f(A)+J)\,...\,(f(a_{l})+j_{l})(f(b_{0})+k_{0})\in nil(f(A)+J)$$
Consequently $(f(a_{i})+j_{i})(f(b_{j})+k_{j})\in nil(f(A)+J)$ when $i+j=l$, and therefore $(f(a_{i})+j_{i})(f(b_{j})+k_{j})\in nil(f(A)+J)$ for every $i, j$.
Hence $(a_{i},f(a_{i})+j_{i})(b_{j},f(b_{j})+k_{j})\in nil(A\bowtie^{f}J)$, and this shows that $A\bowtie^{f}J$ is weak Armendariz.\\

\8 Assume that $f^{-1}(J)$ is semicommutative and $f(A)+J$ is weak Armendariz.
Let $F(x)=\sum_{i=0}^{i=n}(a_{i},f(a_{i})+j_{i})x^{i}$ and $G(x)=\sum_{j=0}^{j=m}(b_{j},f(b_{j})+k_{j})x^{j}$
  be two polynomials in  $A\bowtie^{f}J[x]$ such that $F(x)G(x)=0$ and set  $f_{A}(x)=\sum_{i=0}^{i=n}a_{i}x^{j}$, $g_{A}(x)=\sum_{j=0}^{j=m}b_{j}x^{j}$,
  $f_{B}(x)=\sum_{i=0}^{i=n}(f(a_{i})+j_{i})x^{i}$ and $g_{B}(x)=\sum_{j=0}^{j=m}(f(b_{j})+k_{j})x^{j}$.
  Then $F(x)G(x)=0$ implies that $f_{A}(x)g_{A}(x)=\sum_{l=0}^{l=n+m}(\sum_{i+j=l}a_{i}b_{j})x^{l}=0$
  and $f_{B}(x)g_{B}(x)=\sum_{l=0}^{l=n+m}(\sum_{i+j=l}(f(a_{i})+j_{i})(f(b_{j})+k_{j}))x^{l}=0$.
  Hence $\sum_{i+j=l}a_{i}b_{j}=0 $ for all $l =0,1..., n+m $ and
  $\sum_{i+j=l}(f(a_{i})+j_{i})(f(b_{j})+k_{j})=0$ for all $l =0,1..., n+m$.
  Therefore $(f(a_{i})+j_{i})(f(b_{j})+k_{j})\in nil(f(A)+J)$ since $f(A)+J$ is weak Armendariz.
Since $(f(a_{i})+j_{i})(f(b_{j})+k_{j})=(f(a_{i}b_{j})+t_{ij})\in nil(f(A)+J)$, where $t_{ij}\in J$,
$(f(a_{i}b_{j})+t_{ij})^{n_{ij}}=0$, for some positive integer $n_{ij}$.
Therefore $(f(a_{i}b_{j}))^{n_{ij}}=f((a_{i}b_{j})^{n_{ij}})\in J$, and hence $(a_{i}b_{j})^{n_{ij}}\in f^{-1}(J)$.
Now we show that $a_{i}b_{j}\in nil(A)$ by induction on $i+j$.\\
If $i+j=0$, we have $a_{0}b_{0}=0\in nil(A)$ and so we are done.\\
Let $l$ be a positive integer such that $a_{i}b_{j}\in nil(A)$ when $i+j< l$. As in \7, we will show that $a_{i}b_{j}\in nil(A)$ when $i+j=l$.\\
We have $(a_{0}b_{l})^{n_{0l}=p}\in f^{-1}(J)$ and by the induction hypothesis, $(a_{0}b_{l-1})\in nil(A)$.
Let $t$ be a positive integer such that $(a_{0}b_{l-1})^{t}=0$. Then $(b_{l-1}a_{0} )^{t+1}=0$ and hence\\
$((a_{1}b_{l-1})(a_{0}b_{l})^{p+1}a_{1})(b_{l-1}a_{0} )^{t+1}(b_{l-1}(a_{0}b_{l})^{p+1})=0$.
Since $$(a_{1}b_{l-1})(a_{0}b_{l})^{p+1}a_{1})(b_{l-1}a_{0} )\in  f^{-1}(J)$$
    $$(b_{l-1}a_{0} )^{t}(b_{l-1}(a_{0}b_{l})^{p+1})\in  f^{-1}(J)$$
    $$ (b_{l}(a_{0}b_{l})^{p}a_{1}\in  f^{-1}(J)$$ and $ f^{-1}(J)$ is semicommutative, we obtain\\
    $$((a_{1}b_{l-1})(a_{0}b_{l})^{p+1}a_{1})(b_{l-1}a_{0})(b_{l}(a_{0}b_{l})^{p}a_{1})(b_{l-1}a_{0} )^{t}(b_{l-1}(a_{0}b_{l})^{p+1})=0$$
Hence $[(a_{1}b_{l-1})(a_{0}b_{l})^{p+1}]^{2}a_{1}(b_{l-1}a_{0})^{t}(b_{l-1}(a_{0}b_{l})^{p+1})=0$.\\
iterating this process, we obtain:
    $$[(a_{1}b_{l-1})(a_{0}b_{l})^{p+1}]^{t+3}=0$$
Thus $(a_{1}b_{l-1})(a_{0}b_{l})^{p+1}\in nil f^{-1}(J)$, and similarly we have
     $(a_{i}b_{l-i})(a_{0}b_{l})^{p+1}\in nil f^{-1}(J)$ for $i=2,..., l$.
Since $f^{-1}(J)$ is semicommutative, $nil f^{-1}(J)$ is an ideal and consequently
       $$\sum_{i=1}^{i=l}(a_{i}b_{l-i})(a_{0}b_{l})^{p+1}\in nil f^{-1}(J)$$
Multiply the equation $\sum_{i+j=l}a_{i}b_{j}=0$ on the right side by $(a_{0}b_{l})^{p+1}$, we get:
$$(a_{0}b_{l})^{p+2}=-\sum_{i=1}^{i=l}(a_{i}b_{l-i})(a_{0}b_{l})^{p+1}\in nil f^{-1}(J)$$
Thus $a_{0}b_{l}\in nil(A)$. Now, let $q=n_{1,l-1}$. Then $(a_{1}b_{l-1})^{q}\in f^{-1}(J) $. As in the above proof, we have
$$\sum_{i=2}^{i=l}(a_{i}b_{l-i})(a_{1}b_{l-1})^{q+1}\in nil f^{-1}(J)$$.
Suppose that $(a_{0}b_{l} )^{s}=0$. Then
$$(a_{1}b_{l-1})^{q+1}(a_{0}b_{l})^{s}(a_{1}b_{l-1})^{q+1}=0 $$
Since $(a_{1}b_{l-1})^{q+1}\in f^{-1}(J) $ and $f^{-1}(J)$ is semicommutative,
$$((a_{0}b_{l})(a_{1}b_{l-1})^{q+1})^{s+1}=0 $$
Therefore $$(a_{0}b_{l})(a_{1}b_{l-1})^{q+1}\in nil(f^{-1}(J))$$
If we multiply the equation $\sum_{i+j=l}a_{i}b_{j}=0$ on the right side by
    $(a_{1}b_{l-1})^{q+1}$, we obtain
$$(a_{1}b_{l-1})^{q+2}=-\sum_{i=2}^{i=l}(a_{i}b_{l-i})(a_{1}b_{l-1})^{q+1}-(a_{0}b_{l})(a_{1}b_{l-1})^{q+1}\in nil f^{-1}(J)$$
Therefore $ a_{1}b_{l-1}\in nil(A)$. Similarly, we have $a_{2}b_{l-2}\in nil(A),...,a_{l}b_{0}\in nil(A)$ and consequently $ a_{i}b_{j}\in nil(A))$ when $i+j=l$. Therefore, $a_{i}b_{j}\in nil(A)$ for every $i, j$ and hence $(a_{i},f(a_{i})+j_{i})(b_{j},f(b_{j})+k_{j})\in nil A\bowtie^{f}J$. This shows that $A\bowtie^{f}J$ is weak Armendariz and complete the proof.
\end{proof}

\noindent{\bf Remark}.\\
As we mentioned in the introduction, we do not know so far any example of weak Armendariz ring which is not a
nil-Armendariz ring. This question was left open in \cite{ant}. We were not able to answer the question of whether $A\bowtie^{f}J$ is a nil-Armendariz ring if and only it is a weak Armendariz ring. A negative answer will provide a counter-example of a weak Armendariz ring that is not nil-Armendariz. However, a positive answer shows that amalgamation of algebras along ideals, as a source of examples and counter-examples, cannot provide such example if it exists.
\bibliographystyle{amsplain}

\end{document}